\documentclass[12pt]{amsart}
\usepackage{amssymb}
\usepackage{amsmath}
\usepackage{amsthm}
\usepackage{times}
\usepackage{changepage}
\usepackage{caption}
\usepackage{geometry}
\usepackage[curve,all]{xy}
\xyoption{curve}
\xyoption{line}
\xyoption{arc}
\xyoption{color}
\usepackage{array}
\usepackage{enumitem}
\usepackage{graphicx}
\usepackage[dvipdfmx]{color}
\usepackage{color}
\usepackage{amsthm}
\newtheorem{theorem}{Theorem}[section]
\newtheorem{lemma}[theorem]{Lemma}
\newtheorem{corollary}[theorem]{Corollary}
\newtheorem{proposition}[theorem]{Proposition}

\usepackage{subfig}
\usepackage{hyperref}
 
\theoremstyle{definition}
\newtheorem{definition}[theorem]{Definition}
\newtheorem{example}[theorem]{Example}

\theoremstyle{remark}
\newtheorem{remark}[theorem]{Remark}

\newtheorem{criterion}[theorem]{Criterion}

\numberwithin{equation}{section}

\def\Pic{{\mathrm{Pic}}}

\def\rank{{\text{rank}}}

\def\Hom{{\text{Hom}}}

\def\deg{{\text{deg}}}

\def\I{{\text{I}}}

\def\III{{\text{III}}}

\def\Pic{{\text{Pic}}}

\begin{document}

\title{The 2-divisibility of divisors on K3 surfaces in characteristic 2}

\author{Toshiyuki Katsura}
\thanks{Research of the first author is partially supported by JSPS Grant-in-Aid 
for Scientific Research (C) No.23K03066 and the second author by JSPS
Grant-in-Aid for Scientific Research (A) No.20H00112.
} 

\address{Graduate School of Mathematical Sciences, The University of Tokyo, 
Meguro-ku, Tokyo 153-8914, Japan}
\email{tkatsura@g.ecc.u-tokyo.ac.jp}

\author{Shigeyuki Kond\=o}
\address{Graduate School of Mathematics, Nagoya University, Nagoya,
464-8602, Japan}
\email{kondo@math.nagoya-u.ac.jp}

\date{\today}

\author{Matthias Sch\"utt}
\address{Institut f\"ur Algebraische Geometrie, 
Leibniz Universit\"at  Hannover, Welfengarten 1, 30167 Hannover, Germany, and
\newline\indent
Riemann Center for Geometry and Physics, 
  Leibniz Universit\"at Hannover, 
  Appelstrasse 2, 30167 Hannover, Germany}
\email{schuett@math.uni-hannover.de}

\date{\today}

\begin{abstract}
We show that K3 surfaces in characteristic 2 can admit sets of $n$ disjoint 
smooth rational curves whose sum is divisible by 2 in the Picard group,
for each $n=8,12,16,20$.
More precisely, all values occur on supersingular K3 surfaces,
with  exceptions only at Artin invariants 1 and 10,
while on K3 surfaces of finite height, only $n=8$ is possible.
\end{abstract}

\maketitle

\section{Introduction}
Let $k$ be an algebraically closed field of characteristic $p\geq 0$,
and let $X$ be a K3 surface defined over $k$. 
Let $E_i$ ($i = 1, 2, \ldots, n)$ be disjoint $(-2)$-curves (i.e. non-singular rational curves) on $X$.
In case of $p = 0$, Nikulin \cite{N} showed that if $\sum_{i=1}^{n}E_i$ is divisible by 2
in ${\rm Pic}(X)$, $n$ is equal to either 8 or 16. If $n = 16$, he showed that 
by using $\frac{1}{2}\sum_{i=1}^{16}E_i$ we can construct
a covering $Y$ of degree 2 of $X$ which is an abelian surface. In this case
$X$ is the Kummer surface
associated with the abelian surface $Y$ (cf. Nikulin \cite[Theorem 1]{N}).
This result also holds in characteristic $p \geq 3$ 
(cf. Shepherd-Barron \cite[Remark 3.3 (iii)]{SB}, Matsumoto \cite[Proposition 4.7]{M2}).
But if $p=2$, the situation is
a bit different,
depending on the Picard number $\rho$  and, in the supersingular case $\rho=22$, on the Artin invariant $\sigma$ (see Remark \ref{rem:NS}).

\begin{theorem}
\label{thm}
Let $X$ be a K3 surface in characteristic $p=2$.
Let $n\in {\bf N}$ and assume that $X$ contains disjoint $(-2)$-curves $E_1,\hdots,E_n$
such that 
$\sum_{i=1}^{n}E_i$ is divisible by 2
in ${\rm Pic}(X)$. Then $n\in\{8,12,16,20\}$.
More precisely,
\begin{enumerate}
\item[(i)]
if $\rho(X)<22$, then $n=8$, and this case occurs;
\item[(ii)]
if $\rho(X)=22$, then 
all values for $n$ occur on $X$ except that
\[
n\neq 20 \;\; \text{ if } \; \sigma=1 \;\;\; \text{ and } \;\;\; n\neq 8, 16 \;\; \text{ if } \; \sigma=10.
\]
\end{enumerate}

\end{theorem}

The proof of the theorem proceeds by analysing the structures of $X$ and the covering surface $Y$.
To this end, we  define the conductrix $A$ as in the case of Enriques surfaces (cf.\  Lemma \ref{B=2A} and also \cite{DK}).
In case $A = \emptyset$, we have either $n = 8$ or $n = 12$.
If $n = 8$, then $Y$ contains some  rational singularity while being  birational to a K3 surface, and if $n = 12$, then $Y$ is
nonsingular and birational to an Enriques surface (classical or supersingular, see Proposition \ref{rational}). 
In case $A \neq \emptyset$, then we see $n = 8,12, 16$ or $20$, and 
we show that on $X$ there exists a rational vector field $D$ 
without isolated singular points such that $X^D$ is a nonsingular rational surface
with Picard number $22$ (Theorem \ref{theorem}). The case $n= 20$ is the most interesting case, and 
in this case we make clear the divisibility of the divisor
by using the structure of quasi-elliptic fibration (Corollaries \ref{replace} and \ref{replace2}).
In particular, for quasi-elliptic K3 surfaces with 20 singular fibers of type ${\rm III}$
we determine the divisibility of the divisor $\sum_{i=1}^{20}E_i$ by $2$, where $E_i$'s are
irreducible components of reducible singular fibers (Corollaries \ref{CC'1} and \ref{CC'2}).
We also give  concrete examples of K3 surfaces which have a divisor $\sum_{i=1}^{n}E_i$ divisible by $2$
(Sections \ref{s:ex}, \ref{s:20}).
They are complemented by more abstract arguments in Section \ref{s:pf}
which use lattice theory and genus one fibrations to
complete the proofs of the existence parts of Theorem \ref{thm}.

Note that Y. Matsumoto examines the divisibility by 2
of divisors on K3 surfaces from a different view-point of ours (\cite{M2}).
Namely, he treats the case of $n = 16$ and also includes divisors supported on non-disjoint 
$(-2)$-curves.

\subsection*{Notation} 
A lattice is a free ${\bf Z}$-module $L$ of finite rank with 
a ${\bf Z}$-valued non-degenerate symmetric bilinear form $x\cdot y$ ($x, y \in L$).
We denote by $L^*$ the dual of $L$, that is, $L^*=\Hom(L,{\bf Z})$.  
A lattice is called unimodular if $L \cong L^*$ and even if $x\cdot x$ is even for any $x\in L$.  
A root lattice is a negative definite lattice generated by $(-2)$-vectors.  
We denote by $A_m, D_n$ or $E_k$ $(m\geq 1, n\geq 4, k=6,7,8)$ 
a root lattice defined by the Cartan matrix of the same type.
We denote by $U$ the hyperbolic plane, that is, an even unimodular lattice of signature $(1,1)$.
For lattices $L, M$, $L\oplus M$ is the orthogonal direct sum and $L^{m}$ is 
the orthogonal direct sum of $m$-copies of $L$.
For an even lattice $L$, we define $A_L=L^*/L$ and $q_L: A_L \to {\bf Q}/2{\bf Z}$, 
$q_L(x+L)=x\cdot x + 2{\bf Z}$
which is called the discriminant group and the discriminant quadratic form of $L$, respectively.
We denote by $\ell(L)$ the number of minimal generators of $A_L$.
For an integer $m$, we denote by $L(m)$ the lattice obtained from $L$ by multiplying the bilinear form by $m$.

\subsection*{Acknowledgement}
The authors would like to thank Professor I.~Dolgachev who brought this problem
to their attention.

\section{Dualizing sheaf}

Let $X$ be a non-singular complete algebraic variety over $k$ of characteristic $p >0$ and let
${\mathcal E}$ be a vector bundle of rank 2 over $X$. Assume that we have an exact sequence
$$
   0 \rightarrow {\mathcal M} \longrightarrow {\mathcal E} \longrightarrow {\mathcal L} \rightarrow 0
$$
with line bundles ${\mathcal M}$ and ${\mathcal L}$. Then, as is well-known,
the dual bundle ${\mathcal E}^*$ of ${\mathcal E}$ is given by
\begin{equation}\label{dual}
  {\mathcal E}^*\cong {\mathcal E}\otimes {\mathcal L}^{-1}\otimes {\mathcal M}^{-1}.
\end{equation}
\begin{lemma}\label{torsion-free}
Let $X$ and $Y$ be normal varieties, and let $f : Y \longrightarrow X$ be a finite flat morphism.
Then, $f_{*}({\mathcal O})/{\mathcal O}_X$ is locally free.
\end{lemma}
\begin{proof}
Since $f$ is flat, $f_{*}({\mathcal O})$ is locally free. Therefore,
it suffices to show that $f_{*}({\mathcal O})/{\mathcal O}_X$ has no torsion.
Since the situation is local, we may assume that both $X$ and $Y$ are affine varieties:
$X = {\rm Spec}~A$ and $Y = {\rm Spec}~ B$. We have an inclusion
$A \hookrightarrow B$ and $B$ is integral over $A$.
Suppose that $B/A$ has a torsion $x \in B$. Then, there exists a non-zero element $a \in A$
such that $ax = c \in A$. Then, we have $x = c/a$, and $x$ is contained in
the quotient field of $A$. Since $A$ is normal, we see that $x \in A$.
\end{proof}

Let $X$ and $Y$ be normal complete surfaces, and let $f : Y \longrightarrow X$ be
a finite flat morphism of degree 2. Then, by Lemma \ref{torsion-free}
we have an exact sequence
$$
0 \rightarrow {\mathcal O}_X \longrightarrow f_{*}{\mathcal O}_Y \longrightarrow {\mathcal L} \rightarrow 0
$$
with an invertible sheaf ${\mathcal L}$.
\begin{lemma}\label{dualizing-sheaf}
Under the notation above, we have $\omega_Y \cong f^{*}({\mathcal L}^{-1}\otimes \omega_X)$.
\end{lemma}
\begin{proof}
By the duality theorem of a finite morphism, we have
$$
\begin{array}{rl}
f_{*}{\omega_Y}  &\cong {\mathcal Hom}(f_{*}{\mathcal O}_Y, \omega_X) \cong {\mathcal Hom}(f_{*}{\mathcal O}_Y, {\mathcal O}_X)\otimes \omega_X \\
     &\cong f_{*}{\mathcal O}_Y \otimes {\mathcal L}^{-1}\otimes \omega_X 
     \quad (\mbox{by (\ref{dual}}))\\
  &\cong f_{*}\{{\mathcal O}_Y \otimes f^*({\mathcal L}^{-1}\otimes \omega_X )\} \cong f_{*}(f^*({\mathcal L}^{-1}\otimes \omega_X )).
\end{array}
$$
Since $f$ is a finite morphism, we have 
$\omega_Y \cong f^*({\mathcal L}^{-1}\otimes \omega_X )$.
\end{proof}

\section{$(-2)$-curves on K3 surfaces}
From here on, we assume $p = 2$ unless otherwise mentioned.
We use the notation from the introduction and assume that 
the divisor $\sum_{i=1}^{n}E_i$ is divisible by 2.
Then, there exists an invertible sheaf ${\mathcal L}$ such that
$$
    {\mathcal L}^{\otimes 2}\cong {\mathcal O}_X(\sum_{i=1}^{n}E_i).
$$
In this section we determine the possibility of $n$ and analyse 
the structure of the covering of $X$ constructed by using ${\mathcal L}$.

The self-intersection number of ${\mathcal L}$ is given by
\begin{equation}\label{0}
   {\mathcal L}^2 = -n/2.
\end{equation}
\begin{lemma}\label{4}
$n \equiv 0~({\rm mod}~4)$.
\end{lemma}
\begin{proof}
Since ${\mathcal L}^2$ is an even integer, this follows from (\ref{0}).
\end{proof}
\begin{lemma}\label{morethan8}
$\dim {\rm H}^0(X, {\mathcal L}^{-1}) =\dim {\rm H}^2(X, {\mathcal L}^{-1}) =0$,
$\dim {\rm H}^1(X, {\mathcal L}^{-1}) = \frac{n}{4} -2$.
\end{lemma}
\begin{proof}
By ${\rm H}^0(X, {\mathcal L}^{-2})\cong {\rm H}^0(X, {\mathcal O}_X(-\sum_{i=1}^nE_i)) = 0$,
we have ${\rm H}^0(X, {\mathcal L}^{-1})= 0$. 
If ${\rm H}^0(X, {\mathcal L})\neq 0$, then there exists an effective divisor $D$ such that
$\sum_{i=1}^{n} E_i \sim 2D$. Since $(D, E_i) = -1 < 0$, we see $\sum_{i=1}^{n} E_i \subset D$.
Therefore, the effective divisor $D + (D -\sum_{i=1}^{n} E_i)$ is linearly equivalent to 0,
a contradiction. Therefore, we have ${\rm H}^0(X, {\mathcal L})= 0$, and by the Serre duality
we have ${\rm H}^2(X, {\mathcal L}^{-1})= 0$. Using the Riemann-Roch theorem,
we have
$$
    - \dim {\rm H}^1(X, {\mathcal L}^{-1})= \chi({\mathcal L}^{-1}) = - \frac{n}{4} + 2.
$$
Hence, we complete our proof.
\end{proof}
\begin{corollary}\label{8,12,16,20}
$n = 8, 12, 16$, or $20$.
\end{corollary}
\begin{proof}
$E_i$'s $(i = 1, 2, \ldots, n)$ are linearly independent in N\'eron--Severi group
${\rm NS}(X)$, and the rank of ${\rm NS}(X)$ is at most 22. Therefore, we have $n \leq 22$. 
By Lemma \ref{morethan8}, we have 
$\frac{n}{4} -2 \geq 0$ and by Lemma \ref{4}, we have $n \equiv 0~({\rm mod}~4)$.
Our result follows from these facts.
\end{proof}

Let $\{U_i = {\rm Spec}(A_i)\}$ be an affine open covering of $X$, and let $\{f_{ij}\}$
be a cocycle with respect to the covering $\{U_i\}$ 
which defines ${\mathcal L}$. Let $f_i =0$ be a local equation which defines
the divisor $\sum_{i=1}^{n}E_i$ on $U_i$. 
After adjusting the local equations $f_i$, we may assume that
\begin{equation}\label{1}
       f_{ij}^2 = f_i/f_j \quad \mbox{on}~U_i\cap U_j.
\end{equation}
Let $\tilde{U}_i \longrightarrow U_i$ be the covering defined by
\begin{equation}\label{2}
     z_i^2 = f_i.
\end{equation}
Then patching together by
$$
       z_i = f_{ij}z_j,
$$
we have a purely inseparable covering of degree 2:
$$
    f :  Y \longrightarrow X.
$$
Since $\sum_{i=1}^{n}E_i$ is nonsingular, we see that $Y$ is nonsingular over $E_i$ 
and that by (\ref{2}) we have $f^{-1}E_i = 2\tilde{E}_i$ ($i = 1, 2, \ldots n$).
Considering intersection numbers, we have
$$
    4 \tilde{E}_i^2 = (f^{-1}(E_i))^2 = \deg f\cdot E_i^2 = -4.
$$
Therefore, we have $\tilde{E}_i^2 = -1$, that is, $\tilde{E}_i$'s are 
exceptional curves of the first kind.

We consider a rational 1-form $\eta$ defined by
$$
    \eta = df_i/f_i  ~\mbox{on}~U_i.
$$
Then, this gives a rational one form on $X$. We denote the divisorial part by $(\eta)$
and the isolated singularity part by $\langle \eta \rangle$
(cf. Rudakov-Shafarevich \cite{RS}). Then,
we have
$$
 (\eta) = B - \sum_{i=1}^{n}E_i.
$$
with an effective divisor $B$. 
Since $E_i$'s are nonsingular curves, 
we see that $B$ doesn't intersect the divisor $\sum_{i=1}^{n}E_i$
and that no point in $\langle \eta \rangle$ is contained in $\sum_{i=1}^{n}E_i$.

\begin{lemma}\label{q}
The irregularity $q(Y) = 0$.
\end{lemma}
\begin{proof}
Since $f$ is a purely inseparable morphism of degree 2,  we have a dominant morphism
$X^{(1/2)} \longrightarrow Y$. Considering the universality of Albanese variety, we have
a surjective morphism ${\rm Alb}(X^{(1/2)}) \longrightarrow {\rm Alb}(Y)$.
Since $X^{(1/2)}$ is a K3 surface, we have $\dim {\rm Alb}(X^{(1/2)}) =0$.
Therefore, we have $q(Y) = \dim {\rm Alb}(Y) = 0$.
\end{proof}

Since $f$ is locally given by $A_i \hookrightarrow A_i[z_i]/(z_i^2 - f_i) \cong A_i \oplus A_iz_i$, 
we see that
$$
  f_{*}({\mathcal O}_Y) \cong {\mathcal O}_X \oplus {\mathcal L}^{-1}.
$$
Therefore, by Lemma \ref{dualizing-sheaf} we have 
\begin{equation}\label{3}
\omega_Y\cong f^{*}{\mathcal L}.
\end{equation}
Then, we have
\begin{equation}\label{4}
\omega_Y^{\otimes 2}\cong f^{*}{\mathcal L}^{\otimes 2} \cong f^{*}({\mathcal O}_X(\sum_{i=1}^{n}E_i)) = {\mathcal O}_X(2\sum_{i=1}^{n}\tilde{E}_i)).
\end{equation}
We denote by $Z$ the surface which we get by blowing-down $\tilde{E}_i$'s ($i = 1, 2, \ldots, n$)
on $Y$. Then, we have 
\begin{equation}\label{omega2}
\omega_Z^2 \cong {\mathcal O}_Z.
\end{equation}

\begin{lemma}\label{chi}
$\chi ({\mathcal O}_Z) = \chi ({\mathcal{O}}_Y) = 4 - (n/4)$. 
\end{lemma}
\begin{proof}
Since $f$ is a finite morphism, we see $R^if_{*}({\mathcal O}_Y) = 0$ for $i \geq 1$.
By the Riemann-Roch theorem, we have
$$
\begin{array}{rl}
\chi ({\mathcal O}_Y) &=\chi (f_{*}({\mathcal O}_Y)) = \chi ({\mathcal O}_X \oplus {\mathcal L}^{-1}) 
= \chi ({\mathcal O}_X) + \chi ({\mathcal L}^{-1}) \\
   &= 2 + (2 - (n/4)) = 4 - (n/4).
\end{array}
$$
Since $\tilde{E}_i$'s are exceptional curves of the first kind, 
we have $\chi ({\mathcal O}_Z) = \chi ({\mathcal{O}}_Y)$.
\end{proof}
\begin{corollary}\label{cohomologyY}
$\dim {\rm H}^0(Y, {\mathcal O}_Y)= \dim {\rm H}^2(Y, {\mathcal O}_Y)=1$, 
$\dim {\rm H}^1(Y, {\mathcal O}_Y)=\frac{n}{4} -2$.
\end{corollary}
\begin{proof}
The assertion $\dim {\rm H}^0(Y, {\mathcal O}_Y)= 1$ is clear. Since $f$ is a finite morphism,
we have $R^if_*\omega_Y = 0$ for $i \geq 1$. Therefore, we have
$$
   {\rm H}^2(Y, {\mathcal O}_Y)\cong {\rm H}^0(Y, \omega_Y)\cong 
   {\rm H}^0(X, f_*f^*{\mathcal L}) 
   \cong {\rm H}^0(X, f_*{\mathcal O_Y} \otimes {\mathcal L}) 
   \cong {\rm H}^0(X, {\mathcal L} \oplus {\mathcal O}_X) \cong k.
$$
Finally, by Lemma \ref{chi} we have $\dim {\rm H}^1(Y, {\mathcal O}_Y)=\frac{n}{4} -2$.
\end{proof}
Now, we use the theory of conductrix in \cite[Section 1.3]{CDL}.
Let $\nu: Y' \longrightarrow Y$ be the normalization
and let $\mu : \tilde{Y} \longrightarrow Y'$ be a resolution of singularities:
$$
\begin{array}{ccccc}
\tilde{Y} & \stackrel{\mu}{\longrightarrow} & Y'  & \stackrel{\nu}{\longrightarrow} & Y \\
   &   & f' \searrow &    & \swarrow f \quad \quad\\
      &  &  & X & 
\end{array}
$$
Then, $f' = f\circ \nu$ is a finite and inseparable morphism of degree 2 from a normal surface
onto a smooth surface. Since $Y'$ is normal, $Y'$ is Cohen-Macaulay. Since $X$ is regular,
we see that $f'$ is flat, and so we have an exact sequence
\begin{equation}\label{M}
0 \longrightarrow {\mathcal O}_X \longrightarrow f'_{*}{\mathcal O}_{Y'} \longrightarrow 
{\mathcal M}^{-1} \longrightarrow 0
\end{equation}
with an invertible sheaf ${\mathcal M}$ on $X$.
Therefore, we have a commutative diagram
$$
\begin{array}{ccccccccc}
0 &\rightarrow &{\mathcal O}_X  &\longrightarrow &f_{*}({\mathcal O}_Y) &\longrightarrow&
 {\mathcal L}^{-1} &\rightarrow &0\\
  &     & || & & \downarrow &   & \downarrow & & \\
0& \rightarrow& {\mathcal O}_X &\longrightarrow &f'_{*}({\mathcal O}_{Y'}) &\longrightarrow &
{\mathcal M}^{-1}  &\rightarrow &0.  
\end{array}
$$
Then, there exists a Cartier divisor $A$ such that ${\mathcal M}\cong {\mathcal L}(-A)$ (cf. \cite{CDL}, and also \cite[p.1299]{KK3}).
\begin{definition} 
$A$ is called a conductrix of the covering $f: Y \longrightarrow X$.
\end{definition}
By Lemma \ref{dualizing-sheaf}, we have
\begin{equation}\label{canonical}
\omega_{Y'} \cong f'^{*}{\mathcal M}\cong f'^{*}({\mathcal L}(-A)).
\end{equation}
By the same argument as in \cite[Theorem 3.6]{KK3}, we have the following lemma.
\begin{lemma}\label{B=2A}
      $B = 2A$.
\end{lemma}

Now, we consider a resolution of singularities $\mu : \tilde{Y} \longrightarrow Y'$,
including the case that $Y$ has only isolated singular points (In this case, we have $Y' = Y$
and $A = \emptyset$).
We denote by $\cup_{i = 1}^{k}G_i$ the support of the exceptional divisors
of the resolution $\mu$. Here, $G_i$'s are the irreducible components.
Then, by Hidaka--Watanabe \cite[Lemma 1.1]{HW}, we have
$$
\begin{array}{cl}
     \omega_{\tilde{Y}} &\sim \mu^{*}\omega_{Y'} - \sum_{i= 1}^{k}r_iG_i   \quad (r_i \geq 0)\\
              & \sim \mu^{*}(f \circ \nu)^{*}({\mathcal L}(-A)) - \sum_{i= 1}^{k}r_iG_i.
\end{array}
$$
Since $Y$ is non-singular on $\tilde{E}_i$ which is a exceptional curve of the first kind,
$\mu^*\nu^*({\tilde{E}_i})$ is also an exceptional curve of the first kind.
Blowing down these $n$ exceptional curves on ${\tilde{Y}}$, we have a nonsingular surface
${\tilde{Z}}$. We denote the blowing-down by $\phi : \tilde{Y}\longrightarrow \tilde{Z}$,
and we set $\tilde{A} =\phi(\mu^{*}(f \circ \nu)^{*}A)$ and $\tilde{G}_i =\phi(G_i)$.
Then, the twice dualiging sheaf $\omega_{\tilde{Z}}^2$ of ${\tilde{Z}}$ is given by
$$
      \omega_{\tilde{Z}}^2 =\mathcal{O}_{\tilde{Z}}(-2{\tilde{A}} - \sum_{i= 1}^{k}2r_i\tilde{G}_i)
$$
If $Y$ has a non-rational isolated singular point, then $Y'$ also has 
a non-rational isolated singular point and there exists $i$ such that $r_i > 0$
by Hidaka--Watanabe \cite{HW}. Therefore,
if $Y$ has non-isolated singularity or a non-rational singular point,
we have the pluri-genus $P_m({\tilde{Z}}) = 0$ for $m \geq 1$, and
we have $\kappa (\tilde{Z}) = - \infty$.
Since the Kodaira dimension of $\tilde{Y}$ is the same as the one of $\tilde{Z}$,
we see $\kappa (\tilde{Y}) = - \infty$ in this case.
If $Y$ has only rational double points, then ${\tilde{A}}=\emptyset$ and $r_i=0$ $(i= 1, \ldots, k)$.
Therefore, we have 
${\rm H}^{0}(\tilde{Z}, \omega_{\tilde{Z}}^{2\ell}) \cong k$. Thus, we have 
$\kappa (\tilde{Y}) = \kappa (\tilde{Z}) = 0$.
Summarizing these results, we have the following lemma.

\begin{lemma}\label{kappa}
$\kappa (\tilde{Y}) \leq 0$. In particular, one of the following two cases holds:
\begin{itemize}
\item[$({\rm i})$]
If $Y$ has non-isolated singularity or a non-rational singular point, then $\kappa (\tilde{Y}) = - \infty$.
\item[$({\rm ii})$]
If $Y$ has only rational double points, then $\kappa (\tilde{Y}) = 0$.
\end{itemize}
\end{lemma}

For the morphism  $\mu : \tilde{Y} \longrightarrow Y'$ 
we consider the Leray spectral sequence. Then, we have an exact sequence
\begin{equation}\label{exact}
    0 \rightarrow {\rm H}^{1}(Y', {\mathcal O}_{Y'}) \rightarrow 
    {\rm H}^{1}(\tilde{Y}, {\mathcal O}_{\tilde{Y}}) \rightarrow 
    {\rm H}^{0}(Y', R^1\mu_{*}{\mathcal O}_{\tilde{Y}}) \rightarrow 
    {\rm H}^{2}(Y', {\mathcal O}_{Y'}) \rightarrow 
    {\rm H}^{2}(\tilde{Y}, {\mathcal O}_{\tilde{Y}}). 
\end{equation}
Therefore, we have
\begin{equation}\label{inequality}
\dim {\rm H}^{1}(Y', {\mathcal O}_{Y'}) \leq \dim {\rm H}^{1}(\tilde{Y}, {\mathcal O}_{\tilde{Y}})
\end{equation}

\begin{remark}
\label{rem:NS}
Rudakov--Shafarevich determined the Picard lattices for supersingular K3 surfaces (\cite{RS}).
Using their list, we  gain some knowledge on the divisibility of divisors by 2 
from the viewpoint of lattice theory.
Below we include a few representations of Picard lattices which will prove useful in the sequel.

\begin{table}[ht!]
\begin{center}
\begin{tabular}{ll}
$\sigma =10$: 
& 
$U(2)\oplus \widetilde{A_1^{20}} = U(2)\oplus E_8(2)\oplus  \widetilde{A_1^{12}}$;
\\
$\sigma =9$: 
&
$U\oplus \widetilde{A_1^{20}} = U\oplus E_8(2)\oplus  \widetilde{A_1^{12}}= U(2)\oplus D_4\oplus \widetilde{A_1^{16}} = U(2)\oplus D_4\oplus  E_8(2)\oplus  \widetilde{A_1^{8}}$;
\\
$\sigma =8$
&
$U(2)\oplus D_4 \oplus D_4\oplus \widetilde{A_1^{12}} = U\oplus D_4\oplus  \widetilde{A_1^{16}}= U\oplus D_4\oplus  E_8(2)\oplus  \widetilde{A_1^{8}}$;
\\
$\sigma =7$
&
$U\oplus D_4 \oplus D_4\oplus \widetilde{A_1^{12}} = U\oplus D_4\oplus  \widetilde{A_1^{8}}\oplus  \widetilde{A_1^{8}}$;
\\
$\sigma=6$: 
&
$U\oplus D_8 \oplus \widetilde{A_1^{12}} = U(2)\oplus D_4\oplus D_8\oplus  \widetilde{A_1^{8}}= U(2)\oplus D_4^5$;
\\
$\sigma=5$:
&
 $U\oplus E_8 \oplus \widetilde{A_1^{12}} = U\oplus D_4\oplus D_8\oplus  \widetilde{A_1^{8}}= U\oplus D_4^5$;
\\
$\sigma=4$: 
&
$U\oplus D_4\oplus E_8\oplus  \widetilde{A_1^{8}}$;
\\
$\sigma=3$: 
&
$U\oplus D_4 \oplus D_8\oplus D_8$;
\\
$\sigma=2$: 
&
$U\oplus D_4\oplus D_8\oplus  E_8$;
\\
$\sigma=1$: 
&
$U\oplus D_4 \oplus E_8^2$\\
~\\
\end{tabular}
\caption{Picard lattices of supersingular K3 surfaces with the Artin invariant $\sigma$}
\label{tab}
\end{center}
\end{table}
Here, we use the following notation.
We denote by $e_i$ a basis of each component of $A_1^{4n}$ and put 
$$
\delta_{4n} = \frac{(e_1 + e_2 +\ldots + e_{4n})}{2}.
$$
Then, $\widetilde{A_1^{4n}}$ is the index two overlattice of $A_1^{4n}$ obtained by adjoining $\delta_{4n}$. 
This lattice is of type {\rm I} in the sense of Rudakov--Shafarevich \cite{RS}, that is, 
its discriminant quadratic form takes values in ${\bf Z}/2{\bf Z}$ (as opposed to $\mathbf Q/2\mathbf Z$).
Along the same lines, the Picard lattice of a supersingular K3 surface is even hyperbolic of type {\rm I}
and thus determined uniquely by the Artin invariant $\sigma$.
\end{remark}

\subsection{Examples}
\label{s:ex}

We give 3 examples (for $n=8,12,16$) which appear naturally in the literature.
Examples for $n=20$ are deferred to Section \ref{s:20}.

\begin{example}[$n = 8$]
\label{ex:8}

We consider a genus one fibration of a Kummer surface in \cite{Ka}
which is constructed from the product of two ordinary elliptic curves.
Then, this has two reducible singular fibers $F_1$, $F_2$ of type ${\rm I}_4^{*}$.
Note that this Kummer surface is an ordinary K3 surface.
Let $F$ be a general fiber.
Let $G_j^{(i)}$ ($j = 1, 2, \ldots, 5$) be the multiple curves of $F_i$, 
and let $H_{k}^{(i)}$ ($k = 1, 2, 3, 4$)
be 4 tails of $F_i$. Since $\sum_{i = 1}^2F_i \sim 2F$, we see
$$
\sum_{i=1}^{2}\sum_{j= 1}^{4}H_{j}^{(i)} \sim 2(F - \sum_{i=1}^{2}\sum_{j= 1}^{5}G_j^{(i)}).
$$
This means that $\sum_{i=1}^{2}\sum_{j= 1}^{4}H_{j}^{(i)}$ is divisible by $2$
in ${\rm Pic}(X)$ and gives an example for $n=8$.
\end{example}

\begin{remark}
Note that the same argument applies to any two non-reduced fibers of some genus one fibration
as they always contain exactly 4 irreducible components of odd multiplicity.
Along the same lines, four non-reduced fibers give rise to a configuration of 16 disjoint $(-2)$-curves
whose sum is 2-divisible ($n=16$), see Example \ref{ex3}.
\end{remark}

\begin{example}[$n = 12$]
\label{ex:12}

Let $Y'$ be the supersingular K3 surface with Artin invariant $\sigma = 1$.
$Y'$ contains 42 $(-2)$-curves (cf. \cite{DK}). There exists Enriques surface $X$
whose canonical covering is a K3-like surface $Y$ (with rational singularities)
which is given by
contracting some 12 $(-2)$-curves on $Y'$:
$$
f : Y \longrightarrow X
$$
(cf. \cite{KK1}). We denote these twelve curves by $E_i$ $(i= 1, 2, \ldots, 12)$.
We consider 12 points of X which are the images of singular points of $Y$ by the morphism $f$.
We call these points canonical points on $X$. 
The image of each of 30 $(-2)$-curves on $Y'$ by $f$ passes through 2 canonical points
among 12 ones (cf. \cite{KK1}). This means that each of 30 $(-2)$-curves which are not contracted intersects just 2 $(-2)$-curves among 12 $(-2)$-curves which are contracted.
The intersections are transversal. We set
$$
     D  = \sum_{i=1}^{12} E_i.
$$
The intersection number of $D$ and each of 30 $(-2)$-curves is equal to $2$,
and $(E_i, D) = -2$. Therefore, the intersection numbers of $D$ and 
the 42 $(-2)$-curves are even integers. Since these 42 $(-2)$-curves generate
${\rm Pic}~(Y)$, we see $\frac{1}{2}D \in {{\rm Pic}(Y)}^*$
(the dual lattice of ${\rm Pic}(Y))$. Since the Artin invariant of $Y$ is $1$,
we have ${\rm Pic}(Y)^*/{\rm Pic}(Y)\cong ({\bf Z}/2{\bf Z})^2$.
The discriminant quadratic form 
$q : {\rm Pic}(Y)^*/{\rm Pic}(Y)\longrightarrow {\bf Q}/2{\bf Z}$
agrees with that of $D_4$ (by the decomposition in Table \ref{tab}),
so assume the value $1$ on any non-zero element.
In comparison, since $D^2 = -24$, we have
$$
q\left(\frac{1}{2}D\right) = -24/4 = -6 \equiv 0~\mbox{mod}~2{\bf Z}.
$$
Therefore, we have $\frac{1}{2}D = 0$ in ${\rm Pic}~(Y)^*/{\rm Pic}~(Y)$,
that is, we see that $\frac{1}{2}D$ is contained in ${\rm Pic}~(Y)$.
\end{example}

\begin{remark}
We will reinterpret this example from a different view-point 
and in greater generality in Section \ref{s:pf}.
\end{remark}

\begin{example}[$n = 16$]
\label{ex3}
\label{ex:16}

We consider a genus one fibration of a supersingular K3 surface in \cite{KS} 
which is constructed as a quotient
of 2 cuspidal curves by the group scheme $\mu_2$. It has 4 reducible singular fibers
$F_i$ ($i = 1,2, 3, 4$) of type ${\rm I}_0^{*}$. Let $F$ be a general fiber.
Let $G_i$ be the curve in the center of $F_i$, and let $H_{j}^{(i)}$ ($j = 1, 2,3, 4$)
be 4 tails of $F_i$. Since $\sum_{i = 1}^4F_i \sim 4F$, we see
$$
\sum_{i=1}^{4}\sum_{j= 1}^{4}H_{j}^{(i)} \sim 2(2F - \sum_{i=1}^{4}G_i).
$$
This means that the divisor $\sum_{i=1}^{4}\sum_{j= 1}^{4}H_{j}^{(i)}$ 
is divisible by $2$ in ${\rm Pic}(X)$ and gives an example for $n=16$.
\end{example}

\section{Case $A = \emptyset$}
In this section we assume $A = \emptyset$
and we calculate the singularities of $Y$. We also examine the property
of the surface $\tilde{Y}$ which is obtained by a resolution of singularities of $Y$.

By the assumption $A = \emptyset$, we have $Y' = Y$.
Since $(\eta) = - \sum_{i}^{n}E_i$, we have
$$
    (\eta)^2 = -2n.
$$
Considering Igusa's formula for $X$,
we have
$$
     c_2(X) = 24 = \deg~ \langle \eta\rangle + (\eta) \cdot \omega_X - (\eta)^2 
     =  \deg \langle \eta\rangle + 2n.
$$
Therefore, we have
\begin{equation}\label{Igusa}
   \deg \langle \eta\rangle = 24 - 2n.
\end{equation}

Similar results to the following lemma are also given in 
Matsumoto \cite{M3}, \cite[Lemma 3.6 and Corollary 3.7]{M1},
but we give here an elementary proof for it.
As for the notation of singularities, see Artin \cite{A}.

\begin{lemma}\label{type}
Let $Q \in Y$ be a singular point and set $f(Q) =P$. Let $t, s$ be a system of parameters
of the local ring $O_P$ at $P$ and let the covering $f : Y \longrightarrow X$ be defined 
by a equation $z^2 = f(t, s)$. If $\dim_k k[[t, s]]/(f_t(t, s), f_s(t, s)) \leq 8$, then 
the type of the singular point $Q$ is one of the following:
$$
    A_1, D_{4}^0, D_{6}^0, E_7^0, D_8^0, E_8^0.
$$
Moreover, the dimensions of $k[[t, s]]/(f_t, f_s)$ over $k$ are given respectively by
$$
      1, 4, 6, 7, 8, 8.
$$
\end{lemma}
\begin{proof}
In $k[[t, s]]$, $f(t, s)$ is written as
$$
  f(t, s) = f_1^2t + f_2^2s + f_3^2ts + f_4^2 \quad (f_1, f_2, f_3, f_4 \in k[[t, s]]).
$$
By the change of coordinates : $z \mapsto z + f_4$, we may assume $f_4 = 0$.
Since $Q$ is a singular point, we have $\deg ~f_1 \geq 1$ and $\deg ~f_2 \geq 1$.
If $f_3$ is a unit, then the lowest term of $f(t,s)$ is $ats$ with $a \in k, a\neq 0$.
Therefore, $P$ is of type $A_1$ and $\dim_k k[[t, s]]/(f_t, f_s) = 1$. 

Let $\deg~ f_3 \geq 1$.
Suppose that $\deg ~f_1 \geq 2$ and $\deg ~f_2 \geq 2$.
Then, $1,t, s, t^2, ts, s^2$, two monomials of degree 3 and one monomial of degree 4
are linear independent in $k[[t, s]]/(f_t, f_s)$. Therefore, we have
$\dim_k k[[t, s]]/(f_t, f_s)\geq 9$, a contradiction.
Therefore, at least one of $f_1$ and $f_2$ is of degree 1, and
we have $\deg~ f(t, s) = 3$.

The terms of degree three of $f(t, s)$ is expressed as
$$
a_3t^3 + a_2t^2s + a_1 ts^2 + a_0s^3 = b(t + \alpha_1s)(t + \alpha_2s)(t + \alpha_3s)
$$
with $a_0, a_1, a_2, a_3, b, \alpha_1, \alpha_2, \alpha_3 \in k$. Therefore,
by a suitable change of coordinates, $f(t, s)$ is expressed as one of the following
three normal forms:
\begin{itemize}
\item[$({\rm i})$] $f(t, s) = ts(t+ s) + g_1^2t + g_2^2s + g_3^2ts$,
\item[$({\rm ii})$] $f(t, s) = t^2s + g_1^2t + g_2^2s + g_3^2ts$,
\item[$({\rm iii})$] $f(t, s) = t^3 + g_1^2t + g_2^2s + g_3^2ts$.
\end{itemize}
Here, $g_i \in k[[t, s]]$ $(i = 1,2, 3)$ with $\deg~ g_1 \geq 2$, $\deg~g_2 \geq 2$ and 
$\deg~ g_3 \geq 1$.

In Case $({\rm i})$, $P$ is of type $D_4^{0}$ and $\dim_k k[[t, s]]/(f_t, f_s)= 4$. 

For Case $({\rm ii})$, we have $f_t = g_1^2 + g_3^2s$ and $f_s = t^2 + g_2^2 + g_3^2t$.
Let $\deg ~g_3 = 1$ and $g_3 = a_1t + a_2s +$ higher terms
with $a_1, a_2 \in k$. If $a_2 = 0$, then 
then $ts^i$ $(i \geq 0)$ are linearly independent in $k[[t, s]]/(f_t, f_s)$,
a contradiction. Therefore, we have $a_2 \neq 0$.
By the change of coordinates
$$
        t \mapsto 1, s \mapsto a_1t + a_2s,
$$
$f(t, s)$ is transformed into $t^2s + ts^3 +$ higher terms. This gives 
the singularity of type $D_6^0$ and $\dim_k k[[t, s]]/(f_t, f_s)= 6$.
In case $\deg ~g_3 \geq 2$, to get $\dim_k k[[t, s]]/(f_t, f_s)\leq 8$
we have $\deg~ g_1 = 2$, say $g_1 = a_1t^2 + a_2 ts + a_3s^2 +$ higher terms
with $a_1, a_2, a_3 \in k$.
If $a_3 = 0$, then $\dim_k k[[t, s]]/(f_t, f_s) = \infty$ in a similar reason
to the above and so $a_3 \neq 0$.
We consider the automorphism of $k[[t, s]]$ defined by
$$
     t \mapsto t, s \mapsto s + a_1^2t^3 + a_2^2ts^2.
$$
Then, $f(t, s)$ is transformed into $t^2s + a_3^2ts^4 +$ higher terms. This gives 
the singularity of type $D_8^0$ and $\dim_k k[[t, s]]/(f_t, f_s)= 8$.

For  Case $({\rm iii})$,
we have $f_t = t^2 + g_1^2 + g_3^2s$ and $f_s = g_2^2 + g_3^2t$.
If $\deg~ g_2 \geq 3$ and $\deg~ g_3 \geq 2$, then
$1, t, s, ts, s^2, ts^2, s^3, s^4, s^5$ are linearly independent in $k[[t, s]]/(f_t, f_s)$. Therefore, we have
$\dim_k k[[t, s]]/(f_t, f_s)\geq 9$, a contradiction.
Therefore, we have either $\deg~ g_2 = 2$ or $\deg ~g_3 = 1$.
If $\deg~ g_3 = 1$, by change of coordinates we have
$f(t,s) = t^3 + s^3t +$ higer terms. This gives of type $E_7^0$ and
$\dim_k k[[t, s]]/(f_t, f_s)= 7$. If $\deg ~g_2 = 2$ and $\deg~ g_3 \geq 2$,
we have $g_2 = a_1t^2 + a_2ts + a_3s^2 + $ higher terms with $a_1, a_2, a_3 \in k$.
If $a_3 = 0$, then $ts^i$ $(i \geq 0)$ are linearly independent in $k[[t, s]]/(f_t, f_s)$,
a contradiction. Therefore, we have $a_3 \neq 0$.
We consider the automorphism of $k[[t, s]]$ defined by
$$
     t \mapsto t + a_1^2t^2s + a_2^2s^3, s \mapsto \sqrt[5]{a_3}s.
$$
Then, $f(t, s)$ is transformed into $t^3 + s^5 +$ higher terms. This gives 
the singularity of type $E_8^0$ and $\dim_k k[[t, s]]/(f_t, f_s)= 8$.
Thus, we complete our proof.
\end{proof}



\begin{corollary}\label{rationality}
Under the assumptions in Lemma \ref{type}, the singular point $Q$ is
a rational double point.
\end{corollary}

\begin{lemma}\label{non-rational}
Assume $A = \emptyset$. Then, $Y$ has no non-rational singularity.
\end{lemma}
\begin{proof}
Any singularities are double points by the definition of $Y$.
Suppose that $Y$ has a non-rational singularity.
Then, we have
$\kappa (\tilde{Y}) = -\infty$, and by $q(\tilde{Y}) = q(Y) = 0$,
$\tilde{Y}$ is a rational surface. Therefore, $X$ is a unirational K3 surface.
By (\ref{inequality}) we have $\dim {\rm H}^{1}(Y, {\mathcal O}_{Y}) = 0$.
Since $\dim {\rm H}^{2}(Y, {\mathcal O}_{Y}) = 0$ or $1$ by (\ref{4}), 
we see $n = 8$ or $12$ by Lemma \ref{chi}.
If $n = 12$, by (\ref{Igusa}) we have $\deg \langle \eta\rangle = 0$.
Therefore, $Y$ is nonsingular. 
If $n = 8$, we have 
$\deg \langle \eta\rangle = 8$. 
Therefore, by Corollary \ref{rationality} the singularities of $Y$ are all rational.
\end{proof}

\begin{proposition}\label{rational}
If $A = \emptyset$, then $\tilde{Y}$ is birational to either a K3 surface or an Enriques
surface,
 and the following results hold:
\begin{itemize}
\item[$({\rm i})$] If $\tilde{Y}$ is birational to a K3 surface, then $n = 8$ and $Y$ has some rational double points.
\item[$({\rm ii})$] If $\tilde{Y}$ is birational to an Enriques surface, then $n = 12$ and $Y$ is the nonsingular blow-up
of a classical or supersingular Enriques surface. 
Moreover, $X$ is  unirational (therefore, supersingular).
\end{itemize}
\end{proposition}

\begin{proof}
By $A = \emptyset$, $Y$ has no non-isolated singularity
and by Lemma \ref{non-rational} the singularities are all rational if 
some singular points exist.
If $Y$ has only rational singularities, then by Lemma \ref{kappa} and 
$q(\tilde{Y}) = q(Y) = 0$, we see that the minimal model $\hat Y$ of $\tilde{Y}$ is either a K3 surface 
or an Enriques surface. 
If $\hat{Y}$ is a K3 surface,
then we have $\chi ({\mathcal O}_Y)= \chi ({\mathcal O}_{\tilde{Y}}) = \chi(\mathcal O_{\hat Y}) = 2$.
Therefore, we have $n=8$ by Lemma \ref{chi} and $\deg~ \langle \eta\rangle = 8$.
Thus, $Y$ has some rational singular points.

If $\hat{Y}$ is an Enriques surface,
then we have $\chi ({\mathcal O}_Y)= \chi ({\mathcal O}_{\tilde{Y}}) = \chi(\mathcal O_{\hat Y}) = 1$. 
Therefore, we have $n=12$ by Lemma \ref{chi}. 
As we have seen in the proof of Lemma \ref{non-rational},
this implies that
$Y$ is nonsingular
and $Y = Y'  = \tilde{Y}$. 

Suppose that $\hat{Y}$ is a singular Enriques surface.
Equivalently, there exist a K3 surface W
and an \'etale covering $g : W \longrightarrow \hat Y$. On the other hand,
we have a dominant rational map
$f^{(1/2)} : X^{(1/2)} \longrightarrow Y$.
By our construction of $Y$, $f^{(1/2)}$ is a morphism. We consider a fiber product:
$$
\begin{array}{ccc}
           W \times_{Y}X^{(1/2)}   &   \stackrel{\tilde{g}}{\longrightarrow}    & X^{(1/2)}\\
          \downarrow &             & \quad \downarrow f^{(1/2)}\\
         W & \stackrel{g}{\longrightarrow} & Y 

\end{array}
$$
Since $g$ is an \'etale morphism, $\tilde{g}$ is also an \'etale morphism.
Since a K3 surface is simply connected, $W \times_{Y}X^{(1/2)}$ splits into
two $X^{(1/2)}$'s. Therefore, we have an isomorphism $X^{(1/2)} \cong W$, which is compatible
with the diagram. However, $g$ is separable and $f^{(1/2)}$ is purely inseparable,
a contradiction. Hence, $\hat Y$ is either classical or supersingular.
Since both classical and supersingular Enriques surfaces are unirational (Blass \cite{B}),
we conclude that $X$ is unirational. 
\end{proof}

\begin{corollary}\label{cor:n=8}
Without the assumption $A = \emptyset$, if $X$ is not supersingular ($\rho(X)<22$)
and $\sum_{i=1}^{n}E_i$ is divisible by $2$, then we have $n=8$.
\end{corollary}

\begin{proof}
This follows from Proposition \ref{rational} and Lemma \ref{kappa} (i).
\end{proof}

\begin{corollary}
Assume $A = \emptyset$ and $n=8$.
The types of singular points of $Y$ are one of the following:
$$
8 {\rm A}_1, 4 {\rm A}_1 + D_4^0, 2 {\rm A}_1 + D_6^0,
{\rm A}_1 + E_7^0, 2{\rm D}^0_4, {\rm D}^0_8, {\rm E}^0_8.
$$
\end{corollary}

\begin{proof}
Since we have $\deg \langle \eta \rangle = 8$,
this proposition follows from Lemma \ref{type}.
\end{proof}

\section{Case $A \neq \emptyset$}
In this section we assume $A \neq \emptyset$. We construct a vector field $D$ on $X$
and by using the quotient surface $X^D$ we examine the structure of $X$.

By (\ref{canonical}) we have 
${\rm H}^{2}(Y', {\mathcal O}_{Y'})= 0$.
Since $\kappa(\tilde{Y}) = -\infty$ by Lemma \ref{kappa} (i), we have
$$
\dim {\rm H}^1(\tilde{Y}, \mathcal{O}_{\tilde{Y}}) = q(\tilde{Y}) = q(Y) = 0.
$$
Therefore, by (\ref{inequality}), we have
$$
 {\rm H}^{1}(Y', {\mathcal O}_{Y'}) = 0.
$$
Since $f'$ is a finite morphism, we have $R^if'_{*}{\mathcal O}_{\tilde{Y}} =  0$ for $i \geq 1$.
Therefore we have
\begin{equation}\label{f'{*}}
  {\rm H}^{i}(X, f'_{*}{\mathcal O}_{Y'}) =  {\rm H}^{i}(Y', {\mathcal O}_{Y'})= 0 
  ~\mbox{for}~i\geq 1.
\end{equation}
Since $X$ is a K3 surface, by the exact sequence (\ref{M}) with ${\mathcal M} \cong {\mathcal L}(-A)$
we have a long exact sequence
$$
\begin{array}{ccccccccc}
 0 &\longrightarrow & k & \longrightarrow & k & \longrightarrow & {\rm H}^0(X, {\mathcal L}^{-1}(A))&&\\
 & \longrightarrow & 0 & \longrightarrow & 
 {\rm H}^1(X, f'_{*}{\mathcal O}_{Y'})& \longrightarrow & {\rm H}^1(X, {\mathcal L}^{-1}(A))&&\\
 & \longrightarrow & k & \longrightarrow & 
 {\rm H}^2(X, f'_{*}{\mathcal O}_{Y'})& \longrightarrow & {\rm H}^2(X, {\mathcal L}^{-1}(A))&\longrightarrow & 0,
\end{array}
$$
and by (\ref{f'{*}}) we have
$$
{\rm H}^0(X, {\mathcal L}^{-1}(A)) = 0, ~{\rm H}^1(X, {\mathcal L}^{-1}(A))\cong k,~{\rm H}^2(X, {\mathcal L}^{-1}(A))= 0.
$$
Therefore, we have $\chi ({\mathcal L}^{-1}(A))=-1$. By the Riemann-Roch Theorem,
we have $\chi ({\mathcal L}^{-1}(A)) = \frac{1}{2}(- \frac{n}{2} + A^2) + 2$. Hence,
we have
\begin{equation}\label{nA}
        n - 2A^2 = 12.
\end{equation}
We use Igusa's formula for the rational 1-form $\eta$:
$$
   24 = c_2(X) = \deg\langle \eta \rangle + (\eta)\cdot K_X - (\eta)^2.
$$
Since $K_X \sim 0$ and $(\eta)^2 = 4A^2 - 2n= -24$, we have
$$
         \deg\langle \eta \rangle = 0, ~\mbox{i.e.}~\langle \eta \rangle = 0.
$$
Therefore, $\eta$ has no isolated singular point.

We have an isomorphism
\begin{equation}\label{Theta}
  \Theta_X \otimes \Omega_X^{2} \cong \Omega^1_X.
\end{equation}
Let $x_i$, $y_i$ be local coordinates of the affine open set $U_i$.
Since $X$ is a K3 surface, there exists a regular two form
$$
        \Omega = h_{i}dx_i\wedge dy_i,
$$
and we have $\eta = df_i/f_i = (f_{ix_{i}}dx_i +f_{iy_{i}}dy_i)/f_i$.
Here, $h_i$ is a unit function on $U_i$.
Therefore, using the isomorphism (\ref{Theta}), we have a rational vector field
$$
   D =  \frac{1}{h_if_i}\left(f_{iy_{i}}\frac{\partial}{\partial x_{i}} + 
   f_{ix_{i}}\frac{\partial}{\partial y_{i}}\right) ~\mbox{on}~ U_i.
$$
Since $\langle \eta \rangle =0$ and $(\eta ) = 2A - \sum_{i=1}^{n}E_i$,
we see
$$
    \langle D \rangle = 0, ~ (D) = 2A - \sum_{i=1}^{n}E_i.
$$

Moreover, by direct calculation, we see $D$ is 2-closed.
Since $D$ has no isolated singular point, the quotient surface $X^D$ is nonsingular
and we have a purely inseparable finite morphism of degree 2:
$$
      \pi : X \longrightarrow X^D.
$$
Since $D(f_i) = 0$, the curves $E_i$ are integral with respect to $D$.
Denoting the image of $E_i$ by $G_i$, we have
$$
    2G_i^2 = (\pi^{-1}G_i)^2 = E_i^2 = -2.
$$
Therefore, we have $G_i^2 = -1$. Since the defining equations of $G_i$'s are given by $f_i$'s,
$G_i$ is nonsingular. Therefore, $G_i$ is an exceptional curve of the first kind.

Since $\tilde{U}_i = {\rm Spec} (A_i \oplus A_iz_i)$, we have the Frobenius image 
$$
\tilde{U}_i^{(2)} = {\rm Spec} (A_i^2 \oplus A_i^2z_i^2) = {\rm Spec} (A_i^2 \oplus A_i^2f_i),
$$
and we have a natural morphism
$$
 Y\stackrel{f}{\longrightarrow} X \stackrel{\pi}{\longrightarrow} X^D \stackrel{\varphi}{\longrightarrow} Y^{(2)}.
$$
Since $\varphi\circ \pi \circ f$ is the Frobenius morphism, we have 
$\deg~ \varphi\circ \pi \circ f = 4$. Therefore, $\varphi$ is a finite birational morphism.
Since $X^D$ is non-singular, we see that $\varphi$ is the normalization morphism of $Y^{(2)}$
by the Zariski main theorem. Therefore, the situation of $X^D \longrightarrow Y^{(2)}$
is similar to the situation of $Y' \longrightarrow Y$. Hence, we conclude that $Y'$ is nonsingular
and that $X^D$ is the Frobenius image of $Y'$, i.e., $X^D \cong Y'^{(2)}$. 
Thus, we have the commutative diagram $(*)$.

\begin{figure}[!htb]
\xy
(0,0)*{};
(0,15)*{(*)};
(45,15)*{Y};
(83,8)*{\varphi};
(42,22)*{\nu};
(60,27)*{F};
(60,4)*{F};
(53,17)*{f};
(70,17)*{\pi}
\ar (48,15)*{};(60,15)*{\quad X}
\ar (67,15)*{};(80,15)*{\quad X^D}
\ar (45,30)*{Y'};(45,18)*{}
\ar (80,13)*{};(80,0)*{\quad Y^{(2)}}
\ar (48,30)*{};(75,18)*{}
\ar (48,13)*{};(75,0)*{}
\endxy
\end{figure}

Summarizing these results, we get the following theorem.
\begin{theorem}\label{theorem-1}
Let $X$ be a K3 surface. Under the notation above, we assume that the divisor $\sum_{i=1}^nE_i$
is divisible by $2$ in ${\rm Pic}(X)$ and $A \neq \emptyset$.
Then, on $X$ there exists a 2-closed vector field $D$ without isolated singular points
such that the quotient surface $X^D$ is a nonsingular rational surface
with Picard number $22$. 
\end{theorem}
\begin{proof}
Since $\kappa (X^D) = \kappa (Y'^{(2)}) = \kappa (Y') = - \infty$ and 
$q(X^D) = q(Y') =0$, $X^D$ is a rational surface. 
Since the quotient morphism $\pi : X \longrightarrow X^D$ is radicial,
we see the second Betti number $b_2(X^D) = b_2(X) = 22$.
Therefore, we have $\rho(X^D) = b_2(X^D) = 22$. The remaining statements were already proved.
\end{proof}

\begin{remark}
Since the singularity of $Y$ is not isolated, $X^D$ is not isomorphic to $Y^{(2)}$. 
\end{remark}

\begin{lemma}\label{-1-4}
Let $C$ be a curve on $X^D$ with negative self-intersection number.
Then, the self-intersection number $C^2$ is equal to either $-1$ or $-4$.
\end{lemma} 
\begin{proof}
If a curve $G$ on $X$ has a negative self-intersection number, then we have $G^2 = -2$.
Therefore, we have $(\pi^{-1}(C))^2 = 2 C^2 = -2$ or $-8$.
\end{proof}

\begin{theorem}\label{theorem}
Assume $A \neq \emptyset$ and $n = 20$. 
\begin{itemize}
\item[$({\rm i})$] $Y'$ (therefore, $X^D$) is a non-singular 
rational surface with the Picard number $\rho (Y') = 22$, which is a surface
constructed by 20 blowing-ups of a relatively minimal ruled surface.
\item[$({\rm ii})$] $X$ is a unirational K3 surface and
has a structure of quasi-elliptic surface whose types of reducible singular fibers
are either ${\rm III}$ or ${\rm I}_0^{*}$.
\end{itemize}
\end{theorem}
\begin{proof}
We already showed that $Y'$ is non-singular and the former part of $({\rm i})$
is proved in Theorem \ref{theorem-1}. 
A rational surface with Picard number 2 is a relatively minimal ruled surface.
Therefore, $X^D$ is a surface
constructed by 20 blowing-ups of a relatively minimal ruled surface.
$Y'$ has the same structure as $X^D$. 

Let $\phi : S \longrightarrow {\bf P}^1$  be the relatively minimal ruled surface.
Then, we have a commutative diagram:
$$
\begin{array}{ccccc}
        X  & \stackrel{\pi}{\longrightarrow} & X^D & \stackrel{g}{\longrightarrow} &  S  \\
       \downarrow \phi''    &        &    \downarrow  \phi' &  \searrow     & \downarrow \phi \\
       {\bf P}^1   & \stackrel{g''}{\longrightarrow} & {\bf P}^1  & \stackrel{g'}{\longrightarrow} 
       & {\bf P}^1
\end{array}
$$ 
Here, $g$ is a composition of 20 blowing-ups and $g'$ is an isomorphism.
$g''$ is also an isomorphism. Otherwise, the fibration $\phi'' : X\longrightarrow {\bf P}^1$
becomes ${\bf P}^1$-fibration, which contradicts the fact that $X$ is a K3 surface. 
By Lemma \ref{-1-4}, on a fiber of $\phi$ the number of blowing-ups is either 1 or 4.
Since the fiber of $\phi''$ is a purely inseparable covering of ${\bf P}^1$,
the fibration of $\phi''$ is quasi-elliptic, and the types of reducible 
singular fibers
are either ${\rm III}$ or ${\rm I}_0^{*}$.
\end{proof}

\begin{remark}
Contrary to the case $A=\emptyset$, we highlight that, for $A \neq \emptyset$, 
then all values $n=8, 12, 16$ or $20$ are a priori possible.
Theorem \ref{thm} makes existence more precise.
\end{remark}

\begin{remark}
By Lemma \ref{-1-4}, the self-intersection number of the section 
of $\phi : S \longrightarrow {\bf P}^1$ with non-positive self-intersection
number can only be $-4$, $-3$, $-2$, $-1$ or $0$.
\end{remark}

\section{Case $A \neq \emptyset$ with $n = 20$}

In this section we assume $A \neq \emptyset$ and $n = 20$.
We examine divisors supported on 20 disjoint $(-2)$-curves and determine
cases where they are divisible by two.

Let $X$ be a supersingular K3 surface and let
$\phi'' : X \longrightarrow {\bf P}^1$ be a quasi-elliptic fibration 
with $\ell$ reducible singular fibers
of type ${\rm I}_0^*$ and $20 - 4\ell$ reducible singular fibers of type ${\rm III}$.
Considering the Picard number, we see that this quasi-elliptic surface 
cannot have any other reducible singular fibers.
Let $C_i$ and $C_i'$ ($i = 1, 2, \ldots, 20 - 4\ell$) be two components of 
each reducible singular fiber of type III, respectively. 
Put $F_i = C_i \cup C_i'$. 
For reducible singular fibers $G_i$ $(i = 1, \ldots, \ell)$ of type ${\rm I}_0^*$, 
we set 
$$
G_i = E_{4i -3} +E_{4i -2} +E_{4i -1} +E_{4i} + 2M_i.
$$
Here, $E_i$'s and $M_i$'s are the irreducible components.
We denote by $F$ a general fiber and
by $\xi$ the cusp locus of $\phi'' : X \longrightarrow {\bf P}^1$.
Note that $\xi$ intersects $F_i$ at the singular point of $F_i$,
and $G_i$ on $M_i$ while not intersecting any $E_i$.

\subsection{The case that $\phi'' : X \longrightarrow {\bf P}^1$ has a 2-section $s\neq \xi$}

A 2-section $s$ is by definition a purely inseparable covering of 
the base curve ${\bf P}^1$.
We assume that $s$ intersects $E_{4i -3}$ $(i = 1, \cdots, k)$  with
$(E_{4i -3}\cdot s) = 2$, and intersects $M_i$ $(i = k+1, \cdots , \ell)$
wirh $(M_i \cdot s) = 1$.

\begin{proposition}\label{intersect}
The 2-section $s$ does not intersect $\xi$
on any reducible singular fibers of type ${\rm III}$.
\end{proposition}
\begin{proof}
Suppose that s intersects $\xi$  on $m$ reducible singular 
fibers of type ${\rm III}$.
Then, we have 
\begin{equation}\label{m}
       m \leq s \cdot \xi.
\end{equation}
We may assume these reducible singular fibers are $F_i$ ($1 \leq i \leq m$).
The intersection points on the reducible singular fibers are 
$C_i \cap C_i'$ ($1 \leq i \leq m$).
We may assume that $C'_j$ intersects $s$ 
with $(C'_j \cdot \xi)=2$ and $C_j$ does not intersect $s$
$(m + 1 \leq j \leq 20-4\ell)$.
Since $F$, $s$, $E_i$'s $(i = 1, \ldots, 4\ell)$ and $C_j$'s 
$(\ell + 1 \leq j \leq 20-4\ell)$ are a basis of 
${\rm Pic}~(X) \otimes_{\bf Z} {\bf Q}$,
there exist rational numbers $a$, $b$, $\alpha_i$ and $\beta_i$
such that
$$
\xi = aF + bs + \sum_{i= 1}^{4\ell}\alpha_iE_i + \sum_{j=1}^{20 -4\ell}\beta_jC_j.
$$
By $2 = \xi\cdot F =2b$, we have $b =1$. For $1 \leq i \leq 4k$ and
$i \not\equiv 1~(\mbox{mod}~4)$, we have
$$
   0 = \xi\cdot E_i = -2\alpha_i.
$$
Therefore, we have $\alpha_i = 0$ ($1 \leq i \leq 4k$,
$i \not\equiv 1~(\mbox{mod}~4)$). For $1 \leq i \leq 4k$ and
$i \equiv 1~(\mbox{mod}~4)$, by
$$
0 = \xi\cdot E_i = 2 -2\alpha_i,
$$
we have $\alpha_i = 1$ ($1 \leq i \leq 4k$,
$i \equiv 1~(\mbox{mod}~4)$). 
For $4k + 1 \leq i \leq 4\ell$, since
$$
  0 = \xi \cdot E_i = -2\alpha_i.
$$
we have $\alpha_i = 0$ ($4k + 1 \leq i \leq 4\ell$).
For $1\leq j \leq m$, we have
$$
1 = \xi \cdot C_j = 1 - 2\beta_j.
$$
Therefore, $\beta_i = 0$ ($1\leq j \leq m$).
For $m +1 \leq j \leq 20 - 4\ell$, by
$$
    1 = \xi\cdot C_j= -2\beta_j.
$$
we have $\beta_j = - \frac{1}{2}$ ($m +1 \leq j \leq 20 - 4\ell$).
Summarizing these results, we have
$$
\xi = aF + s + \sum_{i=1}^{k}E_{4i -3} - \frac{1}{2}\sum_{j=m+1}^{20 - 4\ell}C_j.
$$
Considering the self-intersection number of $\xi$, we have
$$
-2 = \xi^2 = -2 -2k +\frac{1}{4}(-2)(20-4\ell -m)+ 4a + 4k,
$$
that is,
\begin{equation}\label{xi}
20 - (8a + 4k + 4\ell) = m.
\end{equation}
By $m \geq 0$, we have
\begin{equation}\label{inequality1}
a + \frac{k + \ell}{2} \leq \frac{5}{2}
\end{equation}
Using (\ref{m}), we have also
\begin{equation}\label{inequality2}
m \leq \xi\cdot s = 2a -2 + 2k.
\end{equation}
By this equation, we see $a \in \frac{1}{2}{\bf Z}$.
Putting (\ref{xi}) into (\ref{inequality2}), we have
$$
   22 \leq 10a + 6k + 4\ell.
$$
Since $k \leq \ell$, we have
\begin{equation}\label{inequality3}
  \frac{22}{10} \leq a + \frac{6k + 4\ell}{10} \leq a + \frac{k + \ell}{2}
\end{equation}
Since $a$ is a half integer, by (\ref{inequality1}) and (\ref{inequality3})
we have $a + \frac{k + \ell}{2} = \frac{5}{2}$ and $m = 0$ by (\ref{xi}).
Hence, $\xi$ and $s$ does not intersect each other 
on any reducible singular fibers of type ${\rm III}$.
\end{proof}

Using this proposition, we may assume that $C_i$'s ($1 \leq i \leq 20-4 \ell$) don't intersect 
the 2-section $s$, and we have the following formula:
\begin{proposition}\label{formula1}
$$
\xi = \frac{5- k - \ell}{2} F + s + \sum_{i=1}^{k}E_{4i -3} - 
\frac{1}{2}\sum_{j=1}^{20 - 4\ell}C_j.
$$
\end{proposition}

As for the divisibility of a divisor by $2$, we have the following result.
\begin{corollary}\label{corollary1}
Under the notation above, $(1 + \ell -k)F + \sum_{i=1}^{20-4\ell}C_i$ is divisible 
by $2$ in ${\rm Pic}~X$.
\end{corollary}
\begin{proof}
This follows from Proposition \ref{formula1}.
\end{proof}

\begin{corollary}\label{replace}
If $\ell -k$ is odd, then $\sum_{i=1}^{20-4\ell}C_i$ is divisible by $2$.
If $\ell -k$ is even, then
if we replace an odd number of $C_i$ by $C_i'$, for example, $C_i$ ($1 \leq i \leq 2m - 1$) 
by $C_i'$ ($1 \leq i \leq 2m - 1$),
then $\sum_{i= 1}^{2m- 1} C_i' + \sum_{i= 2m}^{20- 4\ell}C_i$ is divisible by $2$.
\end{corollary}
\begin{proof}
The former part is trivial. As for the latter part,
since $F = C_i + C_i'$, we have
$$
\sum_{i= 1}^{2m- 1} C_i' + \sum_{i= 2m}^{20- 4\ell}C_i 
= 2s -2\xi+ (2m+4 -k -\ell)F + 2\sum_{i = 1}^{k}E_{4i -3} - 2\sum_{i = 1}^{2m-1}C_i
$$
in ${\rm Pic}(X)$, which we should prove.
\end{proof}

\begin{corollary}\label{CC'1}
Under the notation above, let
$\phi'' : X \longrightarrow {\bf P}^1$ be a quasi-elliptic fibration with 
20 singular fibers of type ${\rm III}$ and a 2-section $s$. 
We consider a divisor
$$
  D = \sum_{i = 1}^{m}C'_i + \sum_{i= m+1}^{20}C_i.
$$
Then, if $m$ is odd, then $D$ is divisible by $2$ in ${\rm Pic}(X)$,
and if $m$ is even, then $D$ is not divisible by $2$ in ${\rm Pic}(X)$.
\end{corollary}
\begin{proof}
Since a K3 surface has no multiple fiber, a fiber $F$ is not divisible
by $2$. This corollary follows from this fact and Corollary \ref{replace}.
\end{proof}

\subsection{The case that $\phi'' : X \longrightarrow {\bf P}^1$ has a section $s$}

We may assume that the section $s$ intersects $E_{4i -3}$ $(i = 1, \cdots, \ell)$
and that $C_i$ ($1 \leq i \leq 20 - 4\ell$)  does not intersect $s$. 
Then, we have $(E_{4i -3}\cdot s)= 1$, $(E_{i}\cdot s)= 0$ 
$(i = 1, \ldots, 4\ell; i\not\equiv1\ ({\rm mod}~4))$, $(C_i\cdot s) = 0$ 
and $(C'_i\cdot s) = 1$.

\begin{proposition}\label{sectiondivisible}
Let $\phi'' : X \longrightarrow {\bf P}^1$ be a quasi-elliptic surface 
with a section $s$.
Under the notation above,
$$
\xi = (4 -\ell)F + 2s + \sum_{i= 1}^{\ell}E_{4i-3} 
-\frac{1}{2} \sum_{j=1}^{20 -4\ell}C_j.
$$
\end{proposition}
\begin{proof}
Since $F$, $s$, $E_i$'s $(i = 1, \ldots, 4\ell)$ and $C_i$'s 
$(1 \leq j \leq 20-4\ell)$ are a basis of 
${\rm Pic}~(X) \otimes_{\bf Z} {\bf Q}$,
there exist rational numbers $a$, $b$, $\alpha_i$ and $\beta_i$
such that
$$
\xi = aF + bs + \sum_{i= 1}^{4\ell}\alpha_iE_i + \sum_{j=1}^{20 -4\ell}\beta_jC_j.
$$
In a similar way to the proof of Proposition \ref{intersect}, we have
$b = 2$, $\alpha_{4i-3} = 1$, $\alpha_i= 0$ 
$(i \not\equiv 1 ~(\mbox{mod}~4))$
and $\beta_i = -1/2$. Considering $\xi^2$, we have $a = 4 - \ell$. 
\end{proof}

\begin{corollary}  
Under the notation above, $\sum_{i=1}^{20- 4\ell}C_i$ is divisible by $2$.
In particular, $\ell\neq 4$.
\end{corollary}

\begin{proof}
This follows from Proposition \ref{sectiondivisible} and from Corollary \ref{8,12,16,20}.
\end{proof}

\begin{corollary}\label{replace2}
If we replace an even number of $C_i$ by $C'_i$, for example, $C_i$ ($1 \leq i \leq 2m$) 
by $C'_i$ ($1 \leq i \leq 2m$),
then $\sum_{i= 1}^{2m} C'_i + \sum_{i= 2m+1}^{20- 4\ell}C_i$ is divisible by $2$.
In particular, $\sum_{i= 1}^{20- 4\ell} C'_i$ is divisible by $2$.
\end{corollary}
\begin{proof}
The proof is similar to the one of Corollary \ref{replace}.
\end{proof}
\begin{corollary}\label{CC'2} 
Under the notation above, let
$\phi'' : X \longrightarrow {\bf P}^1$ be a quasi-elliptic fibration with 
20 reducible singular fibers of type ${\rm III}$ and a section $s$. 
We consider a divisor
$$
  D = \sum_{i = 1}^{m}C'_i + \sum_{i= m+1}^{20}C_i.
$$
Then, if $m$ is odd, then $D$ is not divisible by $2$ in ${\rm Pic}(X)$,
and if $m$ is even, then $D$ is divisible by $2$ in ${\rm Pic}(X)$.
\end{corollary}

\begin{proof}
Since a K3 surface has no multiple fiber, a fiber $F$ is not divisible
by $2$. This corollary follows from this fact and Corollary \ref{replace2}.
\end{proof}
The following corollary is also read off from the list in Shimada \cite[Theorem1.4]{Sh}.

\begin{corollary} 
\label{cor:no}
There exists no genus one fibration on a K3 surface with 4 singular fibers of type ${\rm I}_0^*$
and 4 singular fibers of type ${\rm III}$.
\end{corollary}
\begin{proof}
Suppose that there exists such a genus one fibration $\pi$ on a K3 surface.
Considering the relative Jacobian, we may assume that the fibration has a section.
It follows from Ito \cite[Theorem 2.2]{I2} that there exists no such elliptic fibration.
Suppose now that $\pi$ is quasi-elliptic.
Then, using the notation above, we see the divisor $\sum_{i=1}^{4}C_i$ is divisible by $2$
by Proposition \ref{sectiondivisible}.
On the other hand, the divisor $\sum_{i=1}^{4}C_i$ is not divisible by $2$ by
Corollary \ref{8,12,16,20}, a contradiction.
\end{proof}

\section{Concrete results for $n = 20$}
\label{s:20}


In this section, we develop some fairly concrete results for $n=20$, using the notation of the previous section.
These will prove quite useful in proving Theorem \ref{thm} in Section \ref{s:pf}.

The generic point of the moduli of supersingular K3 surfaces gives
an example of a supersingular K3 surface $X$ 
with 20 reducible singular fibers of type ${\rm III}$ 
with 2-section $s\neq\xi$ (Rudakov-Shafarevich \cite{RS}),
thus admitting a configuration of 20 disjoint $(-2)$-curves whose sum is 2-divisible by Corollary \ref{CC'1}.

Explicitly, such a surface is realized
as follows: Consider the first projection map
$pr_1 : S = {\bf P}^1\times {\bf P}^1 \longrightarrow {\bf P}^1$.
Take twenty suitable general points of $S$ and blow-up at these points.
The obtained surface is $X^D$ and $X$ is a purely inseparable covering
of $X^D$.

\subsection{Quasi-elliptic fibrations with section}

Now, we focus on quasi-elliptic fibrations on a K3 surface $X$ with a section.
By Miyanishi \cite{My} (also see Ito \cite{I}) any such fibration is
expressed as
$$
    y^2 = x^3 + c(t)x + d(t)
$$
with $c(t), d(t) \in k[t]$. By a change of coordinates such as
\begin{eqnarray}
\label{eq:change}
   x =X + g(t)^2, y = Y + g(t)X + h(t)  \quad (g(t), h(t)\in k[t])
\end{eqnarray}
if necessary,
we may assume that $X$ is given by the following form:
\begin{equation}\label{general}
    y^2 = x^3 + (t\varphi(t)^2 + a(t)^2)x + t\psi(t)^2
\end{equation}
with $\varphi(t)$, $\psi(t)$, $a(t)\in k[t]$.
The discriminant $\Delta(t)$ of this quasi-elliptic fibration is given by
$$
   \Delta(t) = (t\varphi (t)^2+ a(t)^2)\varphi(t)^4 + \psi(t)^4.
$$
For \eqref{general} to define  a K3 surface, we may assume that 
$${\rm max}\left\{\left[\frac{\deg (t\varphi(t)^2 + a(t)^2)}{4}\right], 
\left[\frac{\deg (t\psi(t)^2)}{6}\right]\right\} = 1$$ 
and for every root $\alpha$ of $\Delta (t) = 0$,
by using the $(t-\alpha)$-adic valuation $v_{\alpha}$ of $k(t)$ such that 
$v_{\alpha}(t-\alpha) = 1$,  we require
$$
{\rm min}\{v_{\alpha}(t\varphi(t)^2 + a(t)^2) -4, v_{\alpha}(t\psi(t)^2) -6\} <0$$
after any transformation \eqref{eq:change} preserving the general shape of \eqref{general} (cf. Ito \cite{I}).
We denote by $\sigma$ the Artin invariant of $X$ and by $r$ the Mordell-Weil rank
of $X \longrightarrow {\bf P}^1$
(i.e.\ the dimension of $\mathrm{MW}(X)$ over the field ${\bf F}_2$).

\begin{lemma}\label{2-torsion}
\label{lem:P}
Assume $\varphi(t)\neq 0$. Then,
$(x, y) =(\psi(t)^2/\varphi(t)^2, \psi(t)^3/\varphi(t)^3+ a(t)\psi(t)/\varphi(t))$
gives a section $P$ of $X \longrightarrow {\bf P}^1$ (which is 2-torsion).
\end{lemma}
\begin{proof}
Straightforward.
Note that the order follows directly from the group structure of the smooth locus of a general fiber being 
$\mathbb G_a$.
\end{proof}

\subsection{Artin invariants $\sigma=3,4,\hdots,9$}

We start drawing consequences for most Artin invariants.

\begin{lemma}
\label{lem:20xIII}
Let $\sigma\in\{3,4,\hdots,9\}$.
Then $X$ admits a quasi-elliptic fibration with a section and 20 fibers of type $\mathrm{III}$.
\end{lemma}

\begin{proof}
By Shimada \cite[Theorem 1.4]{Sh}, there is a decomposition
$$\Pic(X)
\cong U \oplus L \;\;\; \text{ such that } \;\;\; R(L) = A_1^{20}.
$$
Here $R(L)$ denote the lattice generated by the roots in the negative definite lattice $L$.
Any non-zero isotropic vector endows a K3 surface with a genus one fibration
\begin{eqnarray}
\label{eq:f}
f: X \to {\bf P}^1.
\end{eqnarray}
In addition, the hyperbolic plane implies that this fibration admits a section $O$.
Then the reducible fibers are encoded exactly in the root lattice $R(L)$ (cf.\ \cite{Kondo})
-- up to some ambiguity for $A_1$ and $A_2$,
but this is settled presently by topological considerations.
Namely, if some summand $A_1$ were to correspond to a fiber of type I$_2$,
then the fibration would be elliptic, and the Euler--Poincar\'e characteristic would be at least $20\cdot 2=40$.
As this is absurd, we infer that the fibration is quasi-elliptic with 20 fibers of type III as desired.
\end{proof}

\begin{remark}
A Weierstrass form for the above fibrations  with $\sigma=3$ was recently discovered by N.\ Saito.
\end{remark}

\begin{lemma}
The fibration \eqref{eq:f} 
admits a section which meets every reducible fiber in a component $C_i$ different from the component $C_i'$ met by the zero section.
\end{lemma}

\begin{proof}
With 20 fibers of type III, the discriminant $\Delta$ has 20 simple roots.
Hence $\varphi\not\equiv 0$, and there is the section $P$ from Lemma \ref{lem:P}.
We claim that $(P.O)=3$ always. By inspection, this  fails to hold true if and only if gcd$(\varphi, \psi)\neq 1$.
But in that case, there would be a place $v$ such that $v(\Delta)\geq 4$, i.e.\ a fiber of type I$_0^*$ (at least), a contradiction.

With the given intersection number,
we can apply the theory of Mordell--Weil lattices \cite{MWL} to compute the height of $P$ -- which is zero, for $P$ is torsion:
\[
0 = h(P) = 4 + 2 (P.O) - \sum_v \text{contr}_v(P) \geq 10 - 20\cdot \frac 12.
\]
Here the contraction term at a reducible fiber is non-zero if and only if $P$ meets a different component than $O$.
Thus the statement follows.
\end{proof}

\begin{corollary}
\label{cor:s=3-9}
The divisor $\sum_{i=1}^{20} C_i$ is 2-divisible in $\Pic(X)$.
\end{corollary}

\begin{proof}
As predicted by \cite[Lemma 6.16]{MWL}, we can write
\[
P = O + 5 F - \frac 12\sum_{i=1}^{20} C_i
\]
and deduce the claim.
\end{proof}

\subsection{Artin invariant $\sigma=2$}
\label{ss:s=2}

If $X$ has Artin invariant $\sigma=2$, then it admits a quasi-elliptic fibration with 
a section and reducible fibers of type III (16 times) and I$_0^*$ (once);
this follows from Shimada \cite[Theorem 1.4]{Sh} as in the proof of Lemma \ref{lem:20xIII}.
Therefore $\Delta$ has 16 simple roots and one 4-fold root (and $\varphi\not\equiv 0$).
By considering the formal derivative, we deduce that this is a root of $\varphi$ and thus also of $\psi$.
It follows that gcd$(\varphi, \psi)=l_0$ for a linear polynomial $l_0$ (or $\infty$ with multiplicity one,
but by applying a suitable M\"obius transformation we may assume that neither has $\Delta$ a root at $\infty$ nor do $P$ and $O$ intersect there).
Hence the section $P$ satisfies $(P.O)=2$.
The height returns
\[
0 = h(P) = 4 + 2(P.O) - \sum_v \text{contr}_v(P) = 8 - \underbrace{\begin{cases} 0\\ 1\end{cases}}_{\text{at I}_0^*}
- \underbrace{\sum\begin{cases} 0\\
1/2
\end{cases}}
_{\text{ at the 16 III}}
\]
In fact, we claim that $P$ intersects the fiber of type I$_0^*$ at the same component as $O$.
To see this, assume the contrary. Then $P$ specializes to the cusp $(a|_{l_0=0}, 0)$ 
of the fiber at $l_0=0$ of the Weierstrass model \eqref{general}.
Equivalently, $l_0^3\mid (\psi^2+a \varphi^2)$. But then $l_0^6\mid\Delta$, a contradiction.

Thus the contraction term at the fiber of type I$_0^*$ is zero,
and in consequence, the contraction terms at all fibers of type III are $1/2$.
Hence $P$ meets all these fibers in a component $C_i$ different from the component $C_i'$  met by $O$,
and $P$ can be expressed as
\[
P = O + 4F - \frac 12\sum_{i=1}^{16} C_i.
\]
That is, $n=16$ is supported on $X$, but also $n=20$ by virtue of the 2-divisible divisor
\[
\sum_{j=1}^4 E_j + C_1' + \sum_{i=2}^{16} C_i = - 2M_1+ 2C_1' + \sum_{i=1}^{16} C_i
\]
in the notation of the previous section.

\subsection{Artin invariant $\sigma=1$}

We conclude this section with the following negative result 
for Artin invariant $\sigma=1$ (in line with Theorem \ref{thm}):

\begin{lemma}
\label{lem:s=1}
The supersingular K3 surface $X$ of Artin invariant $\sigma=1$ 
does not support a set of 20 disjoint $(-2)$-curves whose sum is 2-divisible in $\Pic(X)$.
\end{lemma}

\begin{proof}
Assume the contrary.
By Theorem \ref{theorem} (ii), this implies that $X$ admits a quasi-elliptic fibration 
whose reducible fibers have types III and I$_0^*$ only.
But this is impossible by Shimada \cite[Theorem 1.4]{Sh}
(or by Elkies--Sch\"utt \cite{ES}).
\end{proof}

\section{Exemplary Hierarchy of quasi-elliptic fibrations}

For immediate use, we record an interesting hierarchy of quasi-elliptic fibrations;
this will highlight how to move from one Artin invariant stratum to the next
while also providing input for the proof of Theorem \ref{thm} (to be completed in the next section).

\begin{lemma}\label{5D4}
Let $\pi : X \longrightarrow {\bf P}^1$ be a quasi-elliptic K3 surface 
with a section $O$. Assume that the types of reducible singular fibers 
are 5${\rm I}_0^{*}$.
Then, $X$ is given by the following equation.
\begin{eqnarray}
\label{eq:qe}
y^2 = x^3 + a(t)^2x + t\prod_{i=1}^{5}(t - \alpha_i)^2.
\end{eqnarray}
Here, the coefficients $\alpha_i\in k$ are different from each other and $\deg~ a(t) \leq 3$.

Conversely, any such equation gives a quasi-elliptic K3 surface
with 5 reducible singular fibers of type ${\rm I}_0^{*}$.
\end{lemma}
\begin{proof}
For such a quasi-elliptic K3 surface,
by Ito \cite[proposition 4.2]{I} the discriminant is given by the form
$\Delta(t) = g(t)^4$ with a polynominal $g(t)$ of degree 5
which has simple roots. Differentiating both sides of
$$
(t\varphi (t)^2+ a(t)^2)\varphi(t)^4 + \psi(t)^4 = g(t)^4,
$$
we have $\varphi(t)^6 = 0$. Therefore, $\varphi(t)=0$,
and \eqref{general} reduces to \eqref{eq:qe}.
The converse follows from the form of the discriminant.
\end{proof}

We return to the general Weierstrass form \eqref{general} of a quasi-elliptic K3 surface with a section.
Note that the even degree coefficients of $a(t)$ could be eliminated by a transformation of the shape \eqref{eq:change}.
This leaves 12 coefficients; comparing against the 4 standard normalizations (M\"obiius transformations and scaling $x,y$)
confirms that a general member has Artin invariant $\sigma=9$, as these surfaces move in an 8 dimensional family
by Rudakov--Shefarevich \cite{RS}.
Accordingly,   the Mordell--Weil torsion-rank $r$
(i.e.\ the dimension over the field ${\bf F}_2$) of $X \longrightarrow {\bf P}^1$ is generally $1$.

For a convenient representation, let $\alpha_i \in k$ $(i = 1, \ldots, 5;\; \alpha_i \neq \alpha_j$  if $i\neq j)$
and set $\psi(t) = \prod_{i=1}^{5}(t - \alpha_i)$.
Since $\deg ~\varphi(t) \leq 3$, taking arbitrary three different $\alpha_i$,
say, $\alpha_1$, $\alpha_2$ and $\alpha_3$, we may write
$$
 \varphi(t) = a_3(t-\alpha_1)(t-\alpha_2)(t-\alpha_3)+a_2(t-\alpha_1)(t-\alpha_2)
 +a_1(t -\alpha_1) + a_0
$$
for some $a_i\in k$.
Denote by $\ell$ the number of reducible fibers of type ${\rm I}_0^{*}$.
For a quasi-elliptic K3 surface whose reducible singular fibers are 
either of type ${\rm III}$
or of type ${\rm I}_0^{*}$, Ito's formula (cf. Ito \cite[Proposition 4.4]{I})
is given by
\begin{equation}\label{Ito}
    \sigma + r = 10 - \ell.
\end{equation}
First, assume $\varphi(t)$ is prime to $\psi(t)$. Then, since $\Delta(t)$ has
no common roots with $\frac{d}{dt}\Delta(t)$,
$\Delta(t)$ has only simple roots. 
This confirms that there are 20 reducible singular fibers of type ${\rm III}$.

Now, let $a_0 = 0$. Then, $v_{\alpha_1}(\Delta(t)) = 4$ and 
a reducible singular fiber of type ${\rm I}_0^{*}$ appears. 
The types of reducible singular fibers of a general member in this equation
are one ${\rm I}_0^{*}$ and 16 ${\rm III}$. 
Obviously, the number of moduli of this
family is 7, and by Ito's formula we have $\sigma + r = 9$.
Therefore, by Lemma \ref{2-torsion} we have $\sigma = 8$ generally.
We repeat this argument for the cases $a_0=a_1=0$, $a_0=a_1=a_2=0$ and
$a_0=a_1=a_2=a_3=0$, respectively. Then, using Lemma \ref{5D4}, 
we get the first four rows of Table \ref{list}.

\begin{theorem}
\label{thm:strata}
The moduli of supersingular K3 surfaces with Artin invariant $\sigma\leq 9$ can be stratified as in Table \ref{tab2}.
\end{theorem}

\begin{table}[!htb]
\centering
\begin{tabular}{|c|c|c|}
\hline
  Artin invariant $\sigma$  & the types of reducible singular fibers & torsion-rank  $r$ of MW
\\ \hline
 9 &  20 ${\rm III}$   &   1
\\ \hline
8  &   16 ${\rm III}$, ~1 ${\rm I}_0^{*}$  &    1
\\ \hline
7  &   12 ${\rm III}$, ~2 ${\rm I}_0^{*}$ &    1
\\ \hline
6  &  8 ${\rm III}$, ~3 ${\rm I}_0^{*}$  &   1
\\ \hline
$\leq$ 5  & 5 ${\rm I}_0^{*}$  &  $0\leq r = 5-\sigma\leq 4$ 
\\ \hline
\end{tabular}
\caption{Reducible singular fibers and Artin invariants of general member of each stratum}
\label{list}
\label{tab2}
\end{table}

\begin{proof}
For $\sigma\geq 6$, the statement was proved before.
Upon specializing to have 5 fibers of type I$_0^*$,
the two-torsion section $P$ disappears (as $\varphi\equiv 0$, in agreement with Corollary \ref{cor:no}).
This confirms that generally $r=0$ by \eqref{Ito}.
The next strata are obtained by successively imposing more and more sections.
Each step depends on one condition;
a priori, the resulting loci could be empty or degenerate,
but since there is a unique terminal object with $r=4$ by \cite{ES},
the strata are as expected.
\end{proof}

Using the notation in the previous section, Table \ref{tab2} gives rise to concrete examples for $n = 8, 12, 16, 20$
such that the divisor $\sum_{i}^{n}C_i$ is divisible by 2,
as will be exploited in the next section. Another construction of 
a quasi-elliptic surface with 5 singular fibers of type ${\rm I}_0^*$ is given 
in Kond\=o-Sch\"oer \cite{KS}.

\section{Proof of Theorem \ref{thm}}
\label{s:pf}

We conclude this paper by working out the proof of Theorem \ref{thm}.
To start with, the statement about the a priori possible values of $n$ was proved in Corollary \ref{8,12,16,20},
and (i) was covered in Corollary \ref{cor:n=8}.
It remains to prove  Theorem \ref{thm} (ii) for all supersingular K3 surfaces.

\medskip

\noindent
The case $n=20$ was ruled out for $\sigma=1$ in Lemma \ref{lem:s=1}
and realized fairly explicitly for $\sigma=2,\hdots,9$ in Corollary \ref{cor:s=3-9} and Subsection \ref{ss:s=2}.
For the generic supersingular K3 surface with $\sigma=10$, we gave an argument at the beginning of Section \ref{s:20},
but for completeness and later use, we give here a general lattice theoretic argument
based on the following characteristic-free criterion:

\begin{criterion}
\label{crit}
If the Picard lattice of an arbitrary supersingular K3 surface $X$ admits 
an orthogonal decomposition
\[
\Pic(X) \cong U(l) \oplus \widetilde{A_1^{4m}} \oplus L
\]
for some $m, l\in {\bf N}$ and an auxiliary lattice $L$, then $X$ contains a set of $4m$ disjoint $(-2)$-curves
whose sum is 2-divisible.
\end{criterion}

\begin{proof}
As in the proof of Lemma \ref{lem:20xIII}, a primitive isotropic vector $e$ from the summand $U(l)$ induces a genus one fibration
\[
f: X \to {\bf P}^1
\]
admitting a multisection of degree $l$.
This has $4m$ singular fibers of type $\I_2$ or $\III$ (induced by the summand $ \widetilde{A_1^{4m}}$),
plus possibly further reducible fibers given by the root sublattice of $L$.
In practice, exhibiting the genus one fibration may involve some reflections (to make $e$ or $-e$ nef,
eliminating the base locus of the linear system $|e|$ resp.\ $|-e|$),
but regardless of this, the images of the generators of the $A_1$-summands under the reflections are single $(-2)$ curves
featured as fiber components.
Of course, any divisibility property is preserved under reflections, 
hence the claim follows.
\end{proof}

Looking at the representations from Table \ref{tab},
the criterion shows that $n=20$ as well as $n=12$ is realized on \emph{any} supersingular K3 surface
of Artin invariant $\sigma=10$.

\medskip

\noindent
We consider the cases $\sigma=6,7,8,9$.
For these Artin invariants, Criterion \ref{crit} applied to the decompositions from Table \ref{tab} shows directly
that $n=8,12,16$ are all realized except the case $\sigma=6, n=16$. This case is also realized by
Table \ref{tab} and the argument as in Example \ref{ex:8}.

\medskip

\noindent
As for the realization of the values $n=8,16$
for the remaining Artin invariants $\sigma\leq 5$,
we recall from Example \ref{ex:8} and\ Example \ref{ex:16}
that having a genus one fibration with two resp.\ four fibers of type $\I_0^*$ suffices.
But then, by Table \ref{tab2}, for each $\sigma\leq 5$, 
there is a quasi-elliptic fibration which even admits five fibers of type $\I_0^*$.

\medskip

\noindent
Now we treat the case $n=12$.
To cover $n=12$ (uniformly across all $\sigma\leq 9$, in fact),
we pursue the approach from \cite[Proposition 12.32, Lemma 12.33]{MWL}
using inseparable base change of semi-stable rational elliptic surfaces.
In fact, it suffices to consider surfaces with $(12-r)$ singular fibers of type $\I_1$ and one of type 
$\I_{r} \; (r=1,2,\hdots,9)$.
These surfaces form $(9-r)$-dimensional families whose general members have Mordell--Weil rank $r-1$ 
(cf.\ the discussion in \cite[\S 8.9]{MWL}).
In \cite[Lemma 12.33]{MWL} it is proved that any supersingular K3 surface $X$ of Artin invariant $\sigma\leq 9$
arises from such a rational elliptic surface $S$ of Mordell-Weil rank $\sigma-1$ by purely inseparable base change.
Note that $S$ has no 2-torsion section by \cite[Corollary 8.32]{MWL},
but $X$ does (see \cite[p.\ 342]{MWL}). Denoting the section by $P$, 
we find that $P$ meets \emph{every} singular fiber in the component
opposite the component met by the zero section $O$, and that $(P.O)=1$.
In lattice language, the lattice $U\oplus A_1^{12-r}\oplus A_{2r-1}$, accounting for the elliptic fibration and the reducible fibers, 
does not embed primitively into $\Pic(X)$,
but its saturation is obtained by adjoining the class of $P$.
Since $(P.O)=1$, we deduce, as before, that
\[
P = O + 3F - \frac 12 \sum_{i=1}^{12-r} C_i - \frac 12 (\Theta_1+2\Theta_2+\hdots+r\Theta_r+(r-1)\Theta_{r+1}+\hdots+\Theta_{2r-1}),
\]
where the $\Theta_i$ denote the components of the fiber of type $\I_{2r}$ of $X$, 
numbered cyclically such that $O$ meets $\Theta_0$,
and the $C_i$ are the components of the fibers of type $\I_2$ not intersecting $O$.
In particular, it follows that
\[
 \frac 12 \left(\sum_{i=1}^{12-r} C_i + \sum_{j=1}^r \Theta_{2j-1}\right) \in\Pic(X)
 \]
as required.

\medskip

\noindent
Finally, we consider the cases $\sigma=10$ and $n=8,16$, respectively.
To complete the proof of Theorem \ref{thm},
it remains to exclude $n=8, 16$ for $\sigma=10$.
To this end, let $\Pic(X) = U(2) \oplus \widetilde{A_1^{20}}$ 
denote the Picard lattice of a supersingular K3 surface $X$ of Artin invariant $\sigma=10$, as in Table \ref{tab},
and assume that
\[
L = \widetilde{A_1^n}\stackrel\iota\hookrightarrow \Pic(X) \;\;\; (n=8,16).
\]
In case $\iota$ is primitive, then we distinguish two cases involving $L^\perp$, an even hyperbolic lattice of rank $22-n$.

\smallskip

\noindent
If $L\oplus L^\perp = \Pic(X)$, then $L^\perp$ has 
discriminant group $A_{L^\perp}\cong({\bf Z}/2{\bf Z})^{22-n}$.
As an orthogonal summand of $\Pic(X)$, $L^\perp$ is automatically of type I.
Thus we infer that $L^\perp$ is $2$-divisible as an even lattice.
But then the scaled lattice $L^\perp(1/2)$ is even unimodular of signature $(1,21-n)$ which is impossible
for $n=8, 16$.

\smallskip

\noindent
If $L\oplus L^\perp \subsetneq \Pic(X)$, then $L$ and $L^\perp$ glue along a common subgroup of the discriminant groups
\[
({\bf Z}/2{\bf Z})^m \cong H \subset A_L, A_{L^\perp} \;\;\; (m>0)
\]
(which is anti-isometric with respect to the induced intersection pairing).
Denoting by 
$$\jmath: H \hookrightarrow A_L\oplus A_{L^\perp}$$ 
the diagonal embedding
(whose image is isotropic),
we obtain from Nikulin's theory \cite{Nikulin} that
\[
({\bf Z}/2{\bf Z})^{20} \cong A_{\Pic(X)} \cong \jmath(H)^\perp/\jmath(H).
\]
Presently, $A_L$ has $n-2$ generators and $A_L^\perp$ has $22-n$ (equalling its full rank),
but in the quotient $ \jmath(H)^\perp/\jmath(H)$, $m$ of the generators of $A_L$ are deemed superfluous
because modulo $\jmath(H)$ they can be represented by an element of $A_{L^\perp}$.
Hence $A_{\Pic(X)}$  has length (i.e.\ minimum number of generators) $\leq 20-m$, giving the desired contradiction.

\medskip

In case $\iota$ is not primitive,
then the primitive closure $L'\subset\Pic(X)$ has discriminant group $A_{L'}\cong({\bf Z}/2{\bf Z})^{m}$ for some $m<n-2= l(\widetilde{A_1^n})$.
It follows for the lengths that
\[
20 = l(\Pic(X)) \leq l(A_{L'}\oplus A_{L^\perp}) \leq m + \rank(L^\perp) = 22 + (m-n)<20.
\]
Again, this gives a contradiction and thus completes the proof of Theorem \ref{thm}.
\qed

\end{document}